\documentclass[oneside,english]{amsart}
\usepackage[T1]{fontenc}
\usepackage[latin9]{inputenc}
\usepackage{units}
\usepackage{amsthm}
\usepackage{amstext}
\usepackage{amssymb}
\usepackage{graphicx}
\usepackage{esint}

\makeatletter

\newcommand{\noun}[1]{\textsc{#1}}
\providecommand{\tabularnewline}{\\}

\numberwithin{equation}{section}
\numberwithin{figure}{section}
\theoremstyle{plain}
\newtheorem{thm}{\protect\theoremname}
  \theoremstyle{remark}
  \newtheorem{rem}[thm]{\protect\remarkname}

\makeatother

\usepackage{babel}
\makeatletter
\addto\extrasfrench{%
   \providecommand{\fg}{\ifdim\lastskip>\z@\unskip\fi~\frqq}%
}

\makeatother
  \addto\captionsenglish{\renewcommand{\remarkname}{Remark}}
  \addto\captionsenglish{\renewcommand{\theoremname}{Theorem}}
  \addto\captionsfrench{\renewcommand{\remarkname}{Remarque}}
  \addto\captionsfrench{\renewcommand{\theoremname}{Théorème}}
  \providecommand{\remarkname}{Remark}
\providecommand{\theoremname}{Theorem}

\begin{document}

\title{Optimization approach for the Monge-Ampère equation }
\maketitle
\begin{center}
\noun{fethi ben belgacem }%
\footnote{\begin{flushleft}
\textit{Department of Mathematics,  Higher Institute of Computer Sciences and Mathematics of Monastir (ISIMM),
 Avenue de la Korniche - BP 223 - Monastir - 5000, TUNISIA. (fethi.benbelgacem@isimm.rnu.tn).}
\par\end{flushleft}%
}
\par\end{center}
\begin{abstract}
This paper studies the numerical approximation of solution of the
Dirichlet problem for the fully nonlinear Monge-Ampère equation. In
this approach, we take the advantage of reformulation the Monge-Ampère
problem as an optimization problem, to which we associate a well defined
functional whose minimum provides us with the solution to the Monge-Ampère
problem after resolving a Poisson problem by the finite element Galerkin
method. We present some numerical examples, for which a good approximation
is obtained in 68 iterations.


\end{abstract}

\section{Introduction}

In this paper, we give a numerical solution for the following Monge-Ampère
problem
\begin{equation}
\left\{ \begin{array}{cc}
\mbox{det}[D^{2}u]=f(x) & x\in\Omega,\\
u_{|\Gamma}=0, & u\:\textrm{convex},
\end{array}\right.\label{eq}
\end{equation}

where $\Omega$ is a smooth convex and bounded domain in $\mathbb{R}^{2},$
$\left[D^{2}u\right]$ is the Hessian of $u$ and $f\in C^{\infty}(\overline{\Omega}),\textrm{ }f>0.$

Equation (\ref{eq}) belongs to the class of fully nonlinear elliptic
equation. The mathematical analysis of real Monge-Ampère and related
equations has been a source of intense investigations in the last
decades; let us mention the following references ( among many others
and in addition to {[}7{]}, {[}9{]}, {[}15{]}): {[}10{]}, {[}8{]},
{[}17, chapter 4{]}, {[}2{]}, {[}28{]}, {[}11{]}-{[}14{]}. Applications
to Mechanics and Physics can be found in {[}27{]}, {[}4{]}, {[}5{]},
{[}18{]}, {[}24{]}, {[}26{]},{[}31{]}, (see also the references therein).

The numerical approximations of the Monge-Ampère equation as well
as related equations have recently been reported in the literature.
Let us mention the references {[}4{]}, {[}29{]}, {[}39{]}, {[}26 {]},
{[}11{]}, {[}32{]}, {[}25{]},{[}28{]}, {[}33{]}; the method discussed
in {[}11{]}, {[}32{]},{[}25{]} is very geometrical in nature. In contrast
with the method introduced by Dean and Glowinski in {[}19 {]} {[}20{]}
{[}21{]}, which is of the variational type. 

On the existence of smooth solution for (\ref{eq}), we recall that
if $f\in C^{\infty}(\overline{\Omega})$ equation (\ref{eq}) has
a unique strictly convex solution $u\in C^{\infty}(\overline{\Omega})$
(see {[}14{]}).

To obtain a numerical solution for (\ref{eq}), we propose a least-square
formulation of (\ref{eq}). In this approach, we take the advantage
of reformulation of the Monge-Ampère problem as a well defined optimization
problem, to which we associate a well functional whose minimum provides
us with the solution to the Monge-Ampère problem after resolving a
Poisson problem by the finite element Galerkin method. The minimum
is computed by the conjugate gradient method.

The remainder of this article is organized as follows. In section
2, We introduce the optimization problem. In section 3, we discuss
a conjugate gradient algorithm for the resolution of the optimization
problem. The finite element implementation of the above algorithm
is discussed in section 4. Finally, in section 5, we show some numerical
results.

\section{Formulation of the Dirichlet problem for the elliptic Monge-Ampère
equation }

Let $u_{I}$ be the solution of (\ref{eq}). Let $\lambda_{1}$ and
$\lambda_{2}$ be the eigenvalues of the matrix $[D^{2}u_{I}].$ We
have
\[
\left\{ \begin{array}{ccc}
\lambda_{1}+\lambda_{2} & = & \Delta u_{I},\\
\lambda_{1}\lambda_{2} & = & \textrm{det}[D^{2}u_{I}]=f.
\end{array}\right.
\]
Then $\lambda_{1}$ and $\lambda_{2}$ are the solutions of the equation
\[
X^{2}-\Delta u_{I}X+f=0.
\]
So 
\[
(\Delta u_{I})^{2}-4f\geq0.
\]
Then
\[
\Delta u_{I}-2\sqrt{f}\geq0.
\]
 Let us set
\[
\Delta u_{I}-2\sqrt{f}=\widetilde{g}\in C^{\infty}(\overline{\Omega}).
\]
We conclude that $u_{I}$ is solution of the following Dirichlet Poisson
problem 
\[
\mathcal{P}_{\widetilde{g}}\left\{ \begin{array}{c}
\Delta u=2\sqrt{f}+\widetilde{g},\\
u_{|\Gamma}=0.
\end{array}\right.
\]

\begin{flushleft}
To compute $\widetilde{g},$ we consider the least-squares functional
$J$ defined on 
\[
E=\left\{ \varphi\in C^{\infty}(\overline{\Omega}),\:\varphi\geq0\right\} ,
\]
as follows: 
\[
J(g)=\frac{1}{2}\int_{\Omega}\left(\textrm{det}\left(D^{2}u^{g}\right)-f\right)^{2}dx,
\]
where $u^{g}$ is the solution of the Dirichlet Poisson problem 
\[
\mathcal{P}^{g}\left\{ \begin{array}{c}
\Delta u=2\sqrt{f}+g,\\
u_{|\Gamma}=0.
\end{array}\right.
\]
 The minimization problem
\begin{equation}
\left\{ \begin{array}{cc}
\widetilde{g}\in E,\\
J(\widetilde{g})\leq J(g) & \forall g\in E,
\end{array}\right.\label{eq:opti1}
\end{equation}
is thus a least-squares formulation of (\ref{eq}). 
\par\end{flushleft}
\begin{thm}
$u_{I}$ is the strictly convex solution of (\ref{eq}) if and only
if there exist a unique solution $\widetilde{g}$ of (\ref{eq:opti1})
such that $u_{I}=u^{\tilde{g}}.$\end{thm}
\begin{proof}
Since $u_{I}$ is solution of $(\mathcal{P}^{\tilde{g}}),$ we have
$u_{I}=u^{\tilde{g}}.$ So $J(\widetilde{g})=0$ and $\tilde{g}$
is a unique solution of (\ref{eq:opti1}).

Conversely, let $\overline{g}$ be a solution of (\ref{eq:opti1}).
Since (\ref{eq}) has a solution $u_{I},$ we can deduce immediatly
that $J(\widetilde{g})=0$ and so, $J(\overline{g})=0.$ It follows
that 
\[
\left\{ \begin{array}{c}
\textrm{det}[D^{2}u^{\bar{g}}]=f,\\
u{}_{|\Gamma}^{\bar{g}}=0.
\end{array}\right.
\]
We have $\Delta u^{\bar{g}}=2f+\overline{g}>0$ and $\textrm{det}[D^{2}u^{\bar{g}}]>0,$
we can deduce that $u^{\bar{g}}$ is strictly convex and from the
uniqueness of solution for (\ref{eq}) we get $u^{\bar{g}}=u_{I}.$ 
\end{proof}

\section{Iterative solution for the minimisation problem }

\subsection{Description of the algorithm. }

The algorithm we consider to solve the problem (\ref{eq:opti1}) which
is based on the PRP (Polak-Ribière-Polyak {[}36,37{]}) conjugate gradient
method reads: 

Given $g^{0}\in E;$ 

then, for\emph{ $k\geq0,$ $g^{k}$} being known in $E$, solve\emph{
}
\[
\mathcal{P}^{g^{k}}\left\{ \begin{array}{c}
\Delta u=2\sqrt{f}+g^{k},\\
u_{|\Gamma}=0.
\end{array}\right.
\]

\emph{Compute, $\nabla J(g^{k})$, }
\begin{quote}
\emph{If $k\geq1,$ $\beta^{k}=\dfrac{\nabla J(g^{k})^{T}(\nabla J(g^{k})-\nabla J(g^{k-1})}{\left\Vert \nabla J(g^{k-1})\right\Vert _{2}^{2}};$}

\emph{
\[
d^{k}=\left\{ \begin{array}{c}
-\nabla J(g^{0})\quad\quad\quad\qquad\qquad if\, k=0\\
-\nabla J(g^{k})+\beta^{k}d^{k-1}\quad\quad\quad\quad if\, k\geq1;
\end{array}\right.
\]
}

\emph{and update $g^{k}$} \emph{by 
\[
g^{k+1}=g^{k}+\alpha^{k}d^{k}.
\]
}
\end{quote}
Where $\alpha_{k}$ is computed with the Armijo-type line search.

\subsection{Solution of sub-problem $(\mathcal{P}^{g}).$ }

We consider first the variational formulation of $(\mathcal{P}^{g})$
\begin{equation}
\left\{ \begin{array}{ccc}
\textrm{Find} & u^{g}\in H_{0}^{1}(\Omega), & \textrm{such that,}\\
a(u^{g},v)=L(v), & \forall v\in H_{0}^{1}(\Omega),
\end{array}\right.\label{eq:v1}
\end{equation}
where 
\begin{equation}
a(u,v)=\int_{\Omega}\nabla u\nabla vdx\label{eq:e1}
\end{equation}
and 
\begin{equation}
L(v)=-\int_{\Omega}\left(2\sqrt{f}+g\right)vdx,\label{eq:m1}
\end{equation}
$a$ in (\ref{eq:e1}) is coercive on $H_{0}^{1}(\Omega).$ For $f\in L^{2}(\Omega)$we
have $\sqrt{f}\in L^{2}(\Omega).$ Since $\Omega$ is bounded and
for $g\in L^{2}(\Omega),$ $L$ in (\ref{eq:m1}) is continuous, then
by the Lax-Milgram theorem $(\mathcal{P}^{g})_{V}$ has a unique solution
$u^{g}.$

\section{Finite element approximation of the minimization problem}

For simplicity, we assume that $\Omega$ is a bounded polygonal domain
of $\mathbb{R}^{2}.$ Let $\mathcal{T}_{h}$ a finite triangulation
of $\Omega$ (like those discussed in e.g, {[}16{]}).

We introduce a

with $P_{1}$ the space of the two-variable polynomials of degree
$\leq1.$ A function $\varphi$ being given in $H^{2}(\Omega)$ we
denote $\frac{\partial^{2}\varphi}{\partial x_{i}x_{j}}$ by 
\begin{equation}
\int_{\Omega}\frac{\partial^{2}\varphi}{\partial x_{i}^{2}}vdx=-\int_{\Omega}\frac{\partial\varphi}{\partial x_{i}}\frac{\partial v}{\partial x_{i}}dx,\:\forall v\in H_{0}^{1}(\Omega),\:\forall i=1,2,\label{eq:49}
\end{equation}
\begin{equation}
\int_{\Omega}\frac{\partial^{2}\varphi}{\partial x_{1}x_{2}}vdx=-\frac{1}{2}\int_{\Omega}\left[\frac{\partial\varphi}{\partial x_{1}}\frac{\partial v}{\partial x_{2}}+\frac{\partial\varphi}{\partial x_{2}}\frac{\partial v}{\partial x_{1}}\right]dx,\:\forall v\in H_{0}^{1}(\Omega).\label{eq:50}
\end{equation}
Let $\varphi\in V_{h};$ taking advantage of relations (\ref{eq:49})
and (\ref{eq:50}) we define the discrete analogues of the differential
operators $D_{ij}^{2}$ by
\begin{equation}
\left\{ \begin{array}{c}
\forall i=1,2,\: D_{hii}^{2}(\varphi)\in V_{0h},\quad\quad\qquad\qquad\qquad\\
\int_{\Omega}D_{hii}^{2}(\varphi)vdx=-{\displaystyle \int_{\Omega}\frac{\partial\varphi}{\partial x_{i}}\frac{\partial v}{\partial x_{i}}dx,\:\forall v\in V_{0h},}
\end{array}\right.\label{eq:51}
\end{equation}
\begin{equation}
\left\{ \begin{array}{c}
D_{h12}^{2}(\varphi)\in V_{0h},\quad\quad\qquad\qquad\qquad\\
\int_{\Omega}D_{h12}^{2}(\varphi)vdx=-{\displaystyle \frac{1}{2}\int_{\Omega}\left[\frac{\partial\varphi}{\partial x_{1}}\frac{\partial v}{\partial x_{2}}+\frac{\partial\varphi}{\partial x_{2}}\frac{\partial v}{\partial x_{1}}\right]dx,\:\forall v\in V_{0h}.}
\end{array}\right.\label{eq:52}
\end{equation}

To compute the above discrete second order partial derivatives we
will use the trapezoidal rule to evaluate the inegrals in the left
hand sides of (\ref{eq:51}) and (\ref{eq:52}). We consider the set
${\scriptstyle \sum}_{h}$ of the vertices of $\mathcal{T}_{h}$ and
${\scriptstyle {\textstyle }\sum}_{0h}$=$\left\{ P\left|\right.P\in{\scriptstyle \sum}_{h},P\notin\Gamma\right\} .$
We define the integers $N_{h}$ and $N_{0h}$ by $N_{h}=\mbox{Card}({\scriptstyle \sum}_{h})$
and $N_{0h}=\mbox{Card}({\scriptstyle \sum}_{0h}).$ So dim$V_{h}=N_{h}$
and dim$V_{0h}=N_{0h}.$ 

For $P_{k}\in{\scriptstyle \sum}_{h}$ we associate the function $w_{k}$
uniquely defined by 
\[
w_{k}\in V_{h},\, w_{k}(P_{k})=1,\: w_{k}(P_{l})=0,\:\mbox{if}\: l=1,...N_{h,}\: l\neq k.
\]

It is well known (e.g., {[}16{]}) that the sets $\mathfrak{B}_{h}=\left\{ w_{k}\right\} _{k=1}^{N_{h}}$
and $\mathfrak{B}_{0h}=\left\{ w_{k}\right\} _{k=1}^{N_{0h}}$ are
vector bases of $V_{h}$ and $V_{0h},$ respectively.

We denote by $A_{k}$ the area of the polygonal which is the union
of those triangles of $\mathcal{T}_{h}$ which have $P_{k}$ as a
common vertex. By applying the trapezoidal rule to the integrals in
the left hand side of relations (\ref{eq:51}) and (\ref{eq:52})
we obtain:
\begin{equation}
\left\{ \begin{array}{c}
\forall i=1,2,\: D_{hii}^{2}(\varphi)\in V_{0h},\quad\quad\qquad\qquad\qquad\\
D_{hii}^{2}(\varphi)(P_{k})=-{\displaystyle \frac{3}{A_{k}}\int_{\Omega}\frac{\partial\varphi}{\partial x_{i}}\frac{\partial w_{k}}{\partial x_{i}}dx,\:\forall k=1,2,...,N_{0h},}
\end{array}\right.\label{eq:55}
\end{equation}
\begin{equation}
\left\{ \begin{array}{c}
D_{h12}^{2}(\varphi)\left(=D_{h21}^{2}(\varphi)\right)\in V_{0h},\quad\quad\qquad\qquad\qquad\\
D_{h12}^{2}(\varphi)(P_{k})=-{\displaystyle \frac{3}{2A_{k}}\int_{\Omega}\left[\frac{\partial\varphi}{\partial x_{1}}\frac{\partial w_{k}}{\partial x_{2}}+\frac{\partial\varphi}{\partial x_{2}}\frac{\partial w_{k}}{\partial x_{1}}\right]dx,\:\forall k=1,2,...,N_{0h}.}
\end{array}\right.\label{eq:56}
\end{equation}

Computing the integrals in the right hand sides of (\ref{eq:55})
and (\ref{eq:56}) is quite simple since the first order derivatives
of $\varphi$ and $w_{k}$ are piecewise constant.

Taking the above relations into account. We approximate the space
$E$ by 
\[
E_{h}=\left\{ \varphi\in V_{h},\:\varphi\geq0\right\} ,
\]
and then the minimization problem (\ref{eq:opti1}) by 
\[
\left\{ \begin{array}{cc}
\widetilde{g}_{h}\in E_{h},\\
J_{h}(\widetilde{g}_{h})\leq J_{h}(g_{h}) & \forall g_{h}\in E_{h},
\end{array}\right.
\]
 Where
\[
J_{h}(g_{h})=\frac{1}{6}{\displaystyle \sum_{k=1}^{N_{h0}}A_{k}\left|D_{h11}^{2}(u_{h}^{g_{h}})(P_{k})D_{h22}^{2}(u_{h}^{g_{h}})(P_{k})-\left(D_{h12}^{2}(u_{h}^{g_{h}}(P_{k})\right)^{2}-f_{h}(P_{k})\right|^{2}},\:\forall g_{h}\in E_{h},
\]
and $f_{h},$ $\widetilde{g}_{h}$, $g_{h}$ are respectively a continuous
approximations of functions $f,$ $\widetilde{g},$ $g$ and $u_{h}^{g_{h}}$
is the solution of the discret variant of the Dirichlet Poisson problem
$(\mathcal{P}^{g})$.

\subsection{Discrete variant of the algorithm }

We will discuss now the solution of (\ref{eq:opti1}) by a discrete
variant of algorithm 3.1. 

\emph{Given $g_{h}^{0}\in E_{h};$ }

then, for\emph{ $k\geq0,$ $g_{h}^{k}$ being known in $E_{h}$, solve
$\mathcal{P}^{g_{h}^{k}},$ }

\emph{Compute, $\nabla J(g_{h}^{k})$, $\gamma_{h}^{k}=\left\Vert \nabla J(g_{h}^{k})\right\Vert _{2}^{2};$}
\begin{quote}
\emph{If $k\geq1,$ $\beta_{h}^{k}=\gamma_{h}^{k}/\gamma_{h}^{k-1};$}

\emph{
\[
d_{h}^{k}=\left\{ \begin{array}{c}
-\nabla J(g^{0})\quad\quad\quad\qquad\qquad if\, k=0\\
-\nabla J(g_{h}^{k})+\beta_{h}^{k}d_{h}^{k-1}\; if\, k\geq1;
\end{array}\right.
\]
}

and update $g_{h}^{k}$\emph{ by 
\[
g_{h}^{k+1}=g_{h}^{k}+\alpha_{h}^{k}d_{h}^{k}.
\]
}\end{quote}
\begin{rem}
There are many approches for finding an avaible step size $\alpha_{h}^{k}$.
Among them the exact line search is an ideal one, but is cost-consuming
or even impossible to use to find the step size. Some inexact line
searches are sometimes useful and effective in practical computation,
such as Armijo line search {[}1{]}, Goldstein line search and Wolfe
line search {[}24,38{]}.

The Armijo line search is commonly used and easy to implement in practical
computation. 

\textbf{Armijo line search}

Let $s>0$ be a constant, $\rho\in\left(0,1\right)$ and $\mu\in\left(0,1\right).$
Choose $\alpha_{k}$ to be the largest $\alpha$ in $\left\{ s,\: s\rho,\: s\rho^{2},...,\right\} $
such that 
\[
J_{h}(g_{h}^{k})-J_{h}(x_{k}+\alpha d_{h}^{k})\geq-\alpha\mu\nabla J_{h}(g_{h}^{k})^{T}d_{h}^{k}.
\]
However, this line search cannot guarantee the global convergence
of the PRP method and even cannot guarantee $d_{k}$ to be descent
direction of $J$ at $g^{k}.$ 
\end{rem}

\subsubsection{Solution of the sub problem \emph{$\mathcal{P}_{g_{h}^{k}}$ }}

Any sub-problem \emph{$(\mathcal{P}^{g_{h}^{k}}),$ }is equivalent
to a finite dimensional variational linear problem which reads as
follows: Find $u_{h}^{g_{h}}\in V_{0h}$ such that 
\begin{equation}
a(u_{h}^{g_{h}},v_{h})=L(v_{h}),\forall v_{h}\in V_{0h}.\label{eq:ev2}
\end{equation}
By the Lax-milgram theorem we can easily show that (\ref{eq:ev2})
has a unique solution $u_{h}^{g_{h}}\in V_{0h}.$

\section{Numerical experiments}

In this section we are going to apply the method discussed in the
previous section to the solution of some test problems. For all these
test problems we shall assume that $\Omega$ is the unit disk. We
first approximate $\Omega$ by a polygonal domain $\Omega_{h}.$ We
consider $\mathcal{T}_{h}$ a finite triangulation of $\Omega_{h}.$

The \textbf{first test problem} is expressed as follows

\begin{equation}
\left\{ \begin{array}{cc}
\mbox{det}[D^{2}u]=4\left(1+2\left(\left|x\right|^{2}\right)\right)e^{2\left(\left|x\right|^{2}-1\right)} & x\in\Omega,\\
u_{|\Gamma}=0, & u\:\textrm{convex}.
\end{array}\right.\label{test1}
\end{equation}

with $\left|x\right|^{2}=x_{1}^{2}+x_{2}^{2}.$ The exact solution
\textit{\emph{$u\in C^{\infty}(\bar{\Omega})$}} to problem (\ref{test1})
is given by 
\[
u\left(x\right)=e^{\left(\left|x\right|^{2}-1\right)}-1.
\]

\begin{rem}
When computing the approximate solutions of these problems, we stopped
the iterations of the algorithm as soon as $\left|J_{h}(g_{h})\right|\leq10^{-6}.$
\end{rem}
We have discretized the optimization problem associated to the problem
(\ref{test1}). We solved the Poisson problem encountred at each iteration
of the algorithm by a fast Poisson solvers.

We have used as initial guess three different constant values for
$g_{h}^{0}.$  The results obtained after 68 iterations are summarized
in Table 1 (where $u_{h}^{c}$ denotes the computed approximate solution
and $\left\Vert .\right\Vert _{0,\Omega}=\left\Vert .\right\Vert _{L^{2}(\Omega)}$). 

The graph of $u_{h}^{c}$ and its contour plot obtained, for ${\displaystyle h=\nicefrac{1}{128}}$
has been respectively visualized on Figure 2 and Figure 3. 
\begin{table}[h]
\caption{First test problem : Convergence of the approximate solution.}

\begin{tabular}{|c|c|c|}
\hline 
$h$ & $g_{h}^{0}$ & $\left\Vert u-u_{h}^{c}\right\Vert _{0,\Omega}$\tabularnewline
\hline 
\hline 
$\begin{array}{c}
\nicefrac{1}{32}\\
\nicefrac{1}{32}\\
\nicefrac{1}{32}
\end{array}$ & $\begin{array}{c}
0.1\\
0.2\\
0.3
\end{array}$ & $\begin{array}{c}
0.8861\times10^{-4}\\
0.5497\times10^{-4}\\
0.3720\times10^{-4}
\end{array}$\tabularnewline
\hline 
$\begin{array}{c}
\nicefrac{1}{64}\\
\nicefrac{1}{64}\\
\nicefrac{1}{64}
\end{array}$ & $\begin{array}{c}
0.1\\
0.2\\
0.3
\end{array}$ & $\begin{array}{c}
0.3416\times10^{-4}\\
0.9121\times10^{-5}\\
0.7554\times10^{-5}
\end{array}$\tabularnewline
\hline 
$\begin{array}{c}
\nicefrac{1}{128}\\
\nicefrac{1}{128}\\
\nicefrac{1}{128}
\end{array}$ & $\begin{array}{c}
0.1\\
0.2\\
0.3
\end{array}$ & $\begin{array}{c}
0.6305\times10^{-5}\\
0.4981\times10^{-5}\\
0.7203\times10^{-6}
\end{array}$\tabularnewline
\hline 
\end{tabular}
\end{table}

\begin{figure}[h]
\includegraphics[width=7cm,height=5cm]{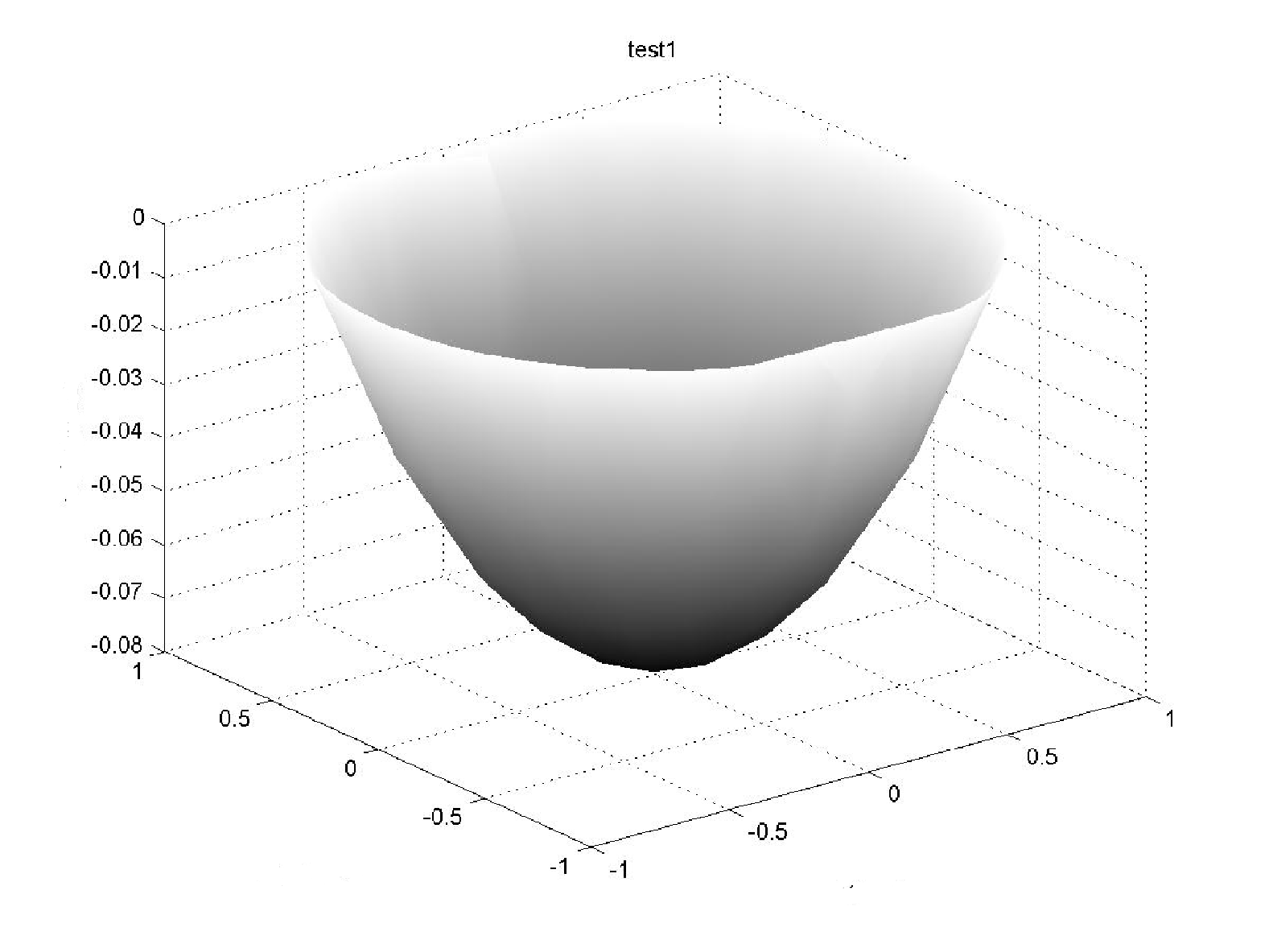}

\caption{First test problem : Graph of $u_{h}^{c}.$}
\end{figure}
\begin{figure}[h]
\includegraphics[width=5cm,height=5cm]{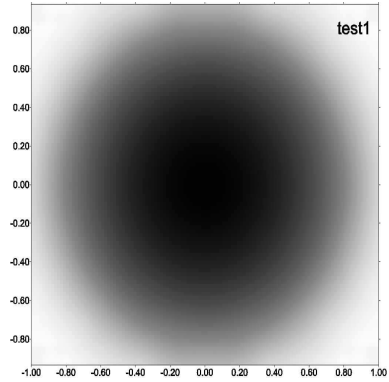}

\caption{First test problem : Contour plot of $u_{h}^{c}.$}
\end{figure}

We conclude from the results in Table 1 that the value $g^{0}=0.3$
is optimal and quite accurate approximations of the exact solutions
are obtained. 
\begin{rem}
We did not try to find the optimal value of $g^{0}$(it seems that
is a difficult problem).

\end{rem}
In\textbf{ the second test problem} we take

\[
f(x)={\displaystyle \left(\frac{4}{5}\right)^{2}\pi^{2}\left[\mbox{cos}^{2}\left(\frac{\pi}{2}\left(1-\left|x\right|^{2}\right)\right)+\frac{\pi}{2}\left(\left|x\right|^{2}\right)\mbox{sin}\left(\pi\left(1-\left|x\right|^{2}\right)\right)\right].}
\]
\foreignlanguage{english}{ The solution to the corresponding Monge-Ampère
problem is the function \textit{\emph{$u\in C^{\infty}(\bar{\Omega})$}}
defined by 
\[
u\left(x\right)=-\frac{4}{5}\mbox{sin}\left(\frac{\pi}{2}\left(1-\left|x\right|^{2}\right)\right).
\]
The method provides after 64 iterations the results summurized in
Table 2.}

\selectlanguage{english}%
The value $g^{0}=0.3$ is again optimal. 

\begin{table}[h]
\caption{second test problem : Convergence of the approximate solution.}

\begin{tabular}{|c|c|c|}
\hline 
$h$ & $g_{h}^{0}$ & $\left\Vert u-u_{h}^{c}\right\Vert _{0,\Omega}$\tabularnewline
\hline 
\hline 
$\begin{array}{c}
\nicefrac{1}{32}\\
\nicefrac{1}{32}\\
\nicefrac{1}{32}
\end{array}$ & $\begin{array}{c}
0.1\\
0.2\\
0.3
\end{array}$ & $\begin{array}{c}
0.6466\times10^{-4}\\
0.4510\times10^{-4}\\
0.2983\times10^{-4}
\end{array}$\tabularnewline
\hline 
$\begin{array}{c}
\nicefrac{1}{64}\\
\nicefrac{1}{64}\\
\nicefrac{1}{64}
\end{array}$ & $\begin{array}{c}
0.1\\
0.2\\
0.3
\end{array}$ & $\begin{array}{c}
0.1749\times10^{-4}\\
0.8507\times10^{-5}\\
0.6221\times10^{-5}
\end{array}$\tabularnewline
\hline 
$\begin{array}{c}
\nicefrac{1}{128}\\
\nicefrac{1}{128}\\
\nicefrac{1}{128}
\end{array}$ & $\begin{array}{c}
0.1\\
0.2\\
0.3
\end{array}$ & $\begin{array}{c}
0.3743\times10^{-5}\\
0.1180\times10^{-5}\\
0.5591\times10^{-6}
\end{array}$\tabularnewline
\hline 
\end{tabular}
\end{table}

\begin{figure*}
\label{c2}\includegraphics[width=7cm,height=5cm]{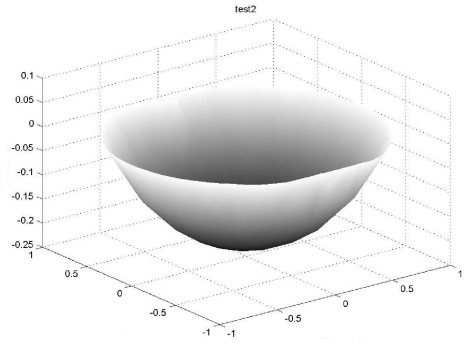}\caption{%
Second test problem : Graph of $u_{h}^{c}.$
}
\end{figure*}
\begin{figure*}
\label{c2}\includegraphics[width=5cm,height=5cm]{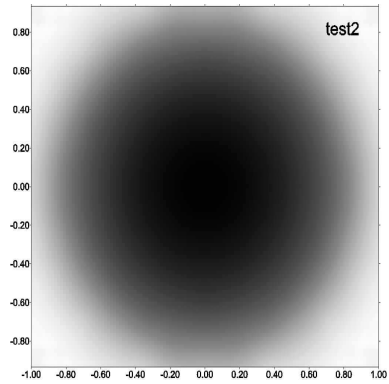}\caption{\selectlanguage{english}%
Second test problem : Contour plot of $u_{h}^{c}.$
}
\end{figure*}

The\textbf{ third test problem} is defined as follows

\begin{equation}
\left\{ \begin{array}{cc}
\mbox{det}[D^{2}u]=1 & x\in\Omega,\\
u_{|\Gamma}=0, & u\:\textrm{convex}.
\end{array}\right.\label{test3}
\end{equation}
The function $u$ given by 

\[
u(x)=\frac{1}{2}\left(\left|x\right|^{2}-1\right)
\]
is the solution of (\textit{\ref{test3}}\textit{\emph{) and  $u\in C^{\infty}(\bar{\Omega}).$}}

\begin{table}[h]
\caption{Third test problem : Convergence of the approximate solution.}

\begin{tabular}{|c|c|c|}
\hline 
$h$ & $g_{h}^{0}$ & $\left\Vert u-u_{h}^{c}\right\Vert _{0,\Omega}$\tabularnewline
\hline 
\hline 
$\begin{array}{c}
\nicefrac{1}{32}\\
\nicefrac{1}{32}\\
\nicefrac{1}{32}
\end{array}$ & $\begin{array}{c}
0.1\\
0.2\\
0.3
\end{array}$ & $\begin{array}{c}
0.3830\times10^{-3}\\
0.2564\times10^{-3}\\
0.2971\times10^{-3}
\end{array}$\tabularnewline
\hline 
$\begin{array}{c}
\nicefrac{1}{64}\\
\nicefrac{1}{64}\\
\nicefrac{1}{64}
\end{array}$ & $\begin{array}{c}
0.1\\
0.2\\
0.3
\end{array}$ & $\begin{array}{c}
0.7448\times10^{-4}\\
0.8529\times10^{-6}\\
0.9193\times10^{-5}
\end{array}$\tabularnewline
\hline 
$\begin{array}{c}
\nicefrac{1}{128}\\
\nicefrac{1}{128}\\
\nicefrac{1}{128}
\end{array}$ & $\begin{array}{c}
0.1\\
0.2\\
0.3
\end{array}$ & $\begin{array}{c}
0.6215\times10^{-5}\\
0.5837\times10^{-6}\\
0.3806\times10^{-5}
\end{array}$\tabularnewline
\hline 
\end{tabular}
\end{table}
 
\begin{figure}[h]
\label{c3}\includegraphics[width=7cm,height=5cm]{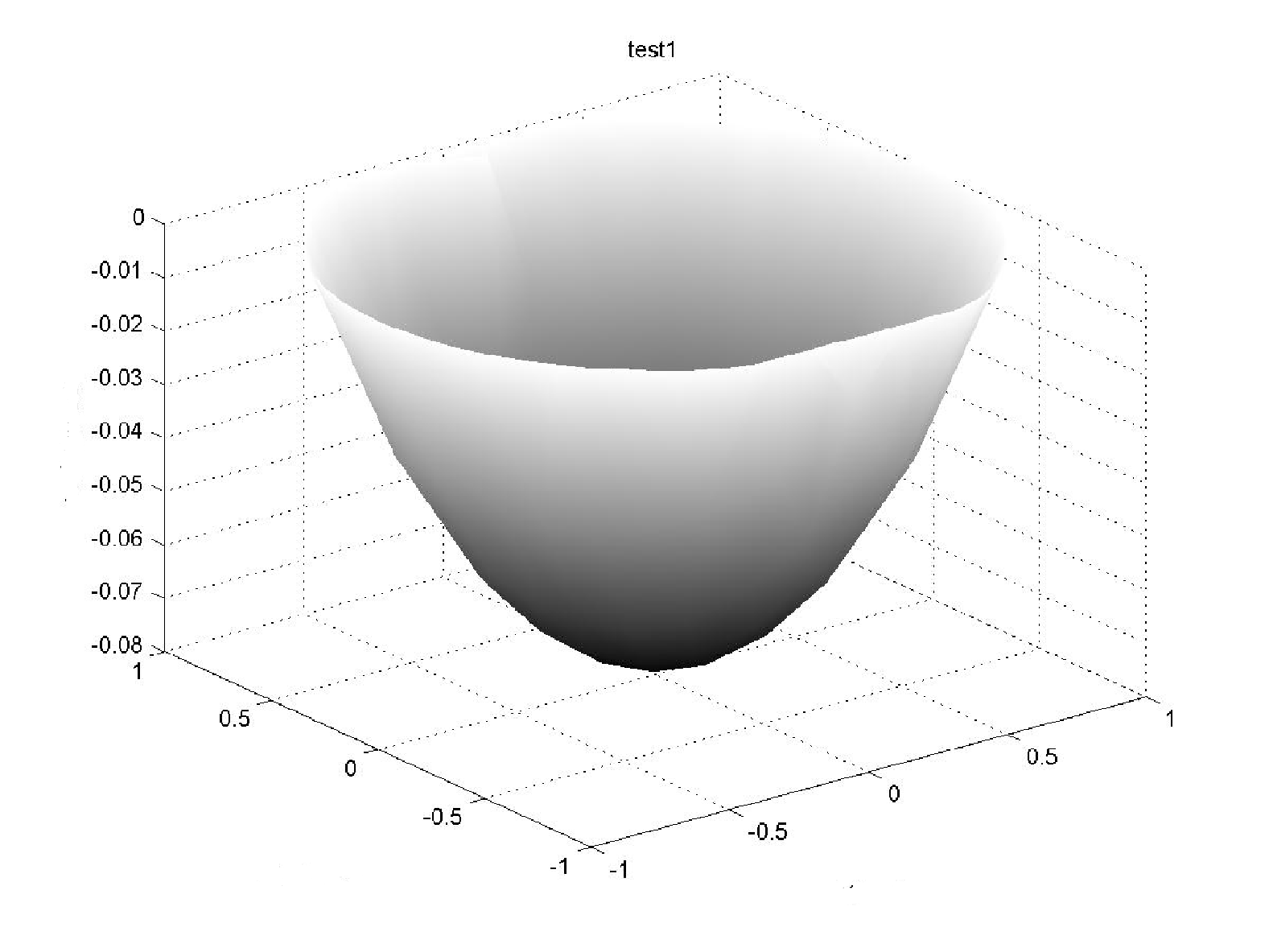}

\caption{Third test problem : Graph of $u_{h}^{c}.$}
\end{figure}
\begin{figure}[h]
\label{c3}\includegraphics[width=5cm,height=5cm]{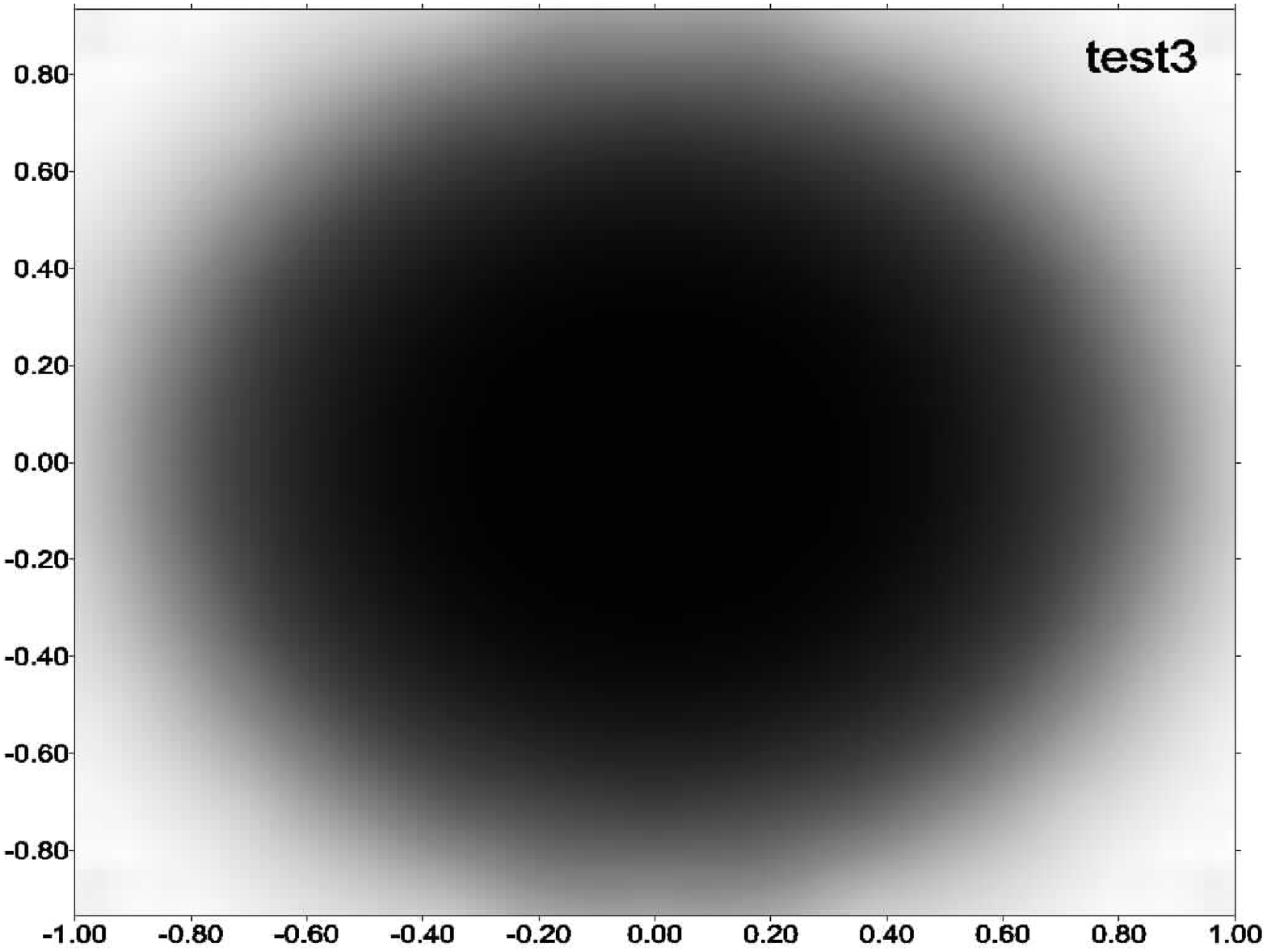}

\caption{Third test problem : Contour plot of $u_{h}^{c}.$}
\end{figure}
\index{}

We deduce from Table 3 that $g_{0}=0.2$ is an optimal value.

Unfortunaly I did not find any other initial value that gives more
accurate results. Even for $g_{0}=0.4$ the results are not satisfied.

\end{document}